\documentclass{amsart}

\usepackage{hyperref}
\hypersetup{nesting=true,debug=true,naturalnames=true}
\usepackage{amssymb,upref}

\usepackage{tikz}

\let\<\langle
\let\>\rangle

\let\uml\"

\usepackage{amsmath,xy,mathrsfs,amsthm,amsxtra,graphicx,verbatim}
\usepackage{hyperref}
\hypersetup{colorlinks=true,urlcolor=magenta,citecolor=magenta,linkcolor=blue}

\usepackage{colonequals}
\usepackage{multirow}

\usepackage[mathcal]{euscript}
\newcommand{\Gal}{\operatorname{Gal}}

\newcommand{\Fq}{\mathbf{F}_{q}}
\newcommand{\Aut}{\operatorname{Aut}}

\newcommand{\Zl}{\mathbf{Z}_{\ell}}

\newcommand{\Z}{\mathbf{Z}}
\newcommand{\F}{\mathbf{F}}

\newcommand{\Q}{\mathbf{Q}}

\newcommand{\GL}{\operatorname{GL}}

\newcommand{\disc}{\operatorname{disc}}

\newcommand{\tors}{\mathsf{tors}}

\newcommand{\Prop}{\operatorname{Prop}}

\newcommand{\legen}[2]{\ensuremath{\left( \frac{#1}{#2} \right) }}

\numberwithin{equation}{section}

\theoremstyle{plain}
\newtheorem{thm}[equation]{Theorem}

\newtheorem{lem}[equation]{Lemma}
\newtheorem{defn}[equation]{Definition}

\newtheorem{prop}[equation]{Proposition}

\newtheorem{qu}{Question}

\theoremstyle{remark}
\newtheorem{rmk}[equation]{Remark}

\newtheorem{exm}[equation]{Example}

 \input xy
 \xyoption{all}

\begin{document}

\title[Isomorphic Group Structures]{The probability of isomorphic group structures of isogenous elliptic curves over finite fields}

\author{John Cullinan}
\address{Department of Mathematics, Bard College, Annandale-On-Hudson, NY 12504, USA}
\email{cullinan@bard.edu}
\urladdr{\url{http://faculty.bard.edu/cullinan/}}

\author{Nathan Kaplan}
\address{Department of Mathematics, University of California, Irvine, CA 92697, USA}
\email{nckaplan@math.uci.edu }
\urladdr{\url{https://www.math.uci.edu/~nckaplan/}}

\begin{abstract}
Let $\ell$ be a prime number and let $E$ and $E'$ be $\ell$-isogenous elliptic curves defined over $\Q$.  In this paper we determine the proportion of primes $p$ for which $E(\F_p) \simeq E'(\F_p)$.  Our techniques are based on those developed in \cite{ck} and \cite{rnt}.
\end{abstract}

\maketitle

\section{Introduction} \label{intro}

\subsection{Preliminaries} Let $E$ and $E'$ be isogenous elliptic curves defined over a finite field $k$.  In a series of papers \cite{cullinan1, cullinan2, ck, rnt} we considered variants of the following question.

\begin{qu}
If $E(k) \simeq E'(k)$, is it true that $E(K) \simeq E'(K)$ as $K$ varies over finite extensions of $k$? 
\end{qu}

The answer is no, and it is not hard to come up with explicit counterexamples.  However, in  \cite{cullinan1} the first-named author showed that if $\ell$ is an odd prime number and $E$ and $E'$ are $\ell$-isogenous, where the kernel is generated by a $k$-rational point of order $\ell$, then $E(k) \simeq E'(k)$ implies that $E(K) \simeq E'(K)$ for all finite extensions $K$ of $k$.   If $\ell=2$ then it was shown in \cite{cullinan2} if $E(k) \simeq E'(k)$ \emph{and} $E(\mathbf{k}) \simeq E'(\mathbf{k})$, where $\mathbf{k}$ is the unique quadratic extension of $k$, then $E(K) \simeq E'(K)$ for all finite extensions $K$ of $k$.

Building on this, in \cite{ck} and \cite{rnt} we fixed 2-isogenous elliptic curves $E$ and $E'$ over $\Q$ and defined a prime $p$ to be \emph{anomalous} for the pair $(E,E')$ if $E(\F_p) \simeq E'(\F_p)$, but $E(\F_{p^2}) \not \simeq E'(\F_{p^2})$.  The main result of these papers is that the set of anomalous primes has a well-defined density that can be realized by summing a certain geometric series associated to the 2-adic representations of $E$ and $E'$.  For example, if the 2-adic representations of $E$ and $E'$ are as large as possible (since they are $2$-isogenous, each has a rational $2$-torsion point, so the images have index $3$ in $\GL_{2}(\Z_2)$), then the anomalous primes for $(E,E')$ have density 1/30. 

In the course of determining the density of anomalous primes for all pairs $(E,E')$ of 2-isogenous elliptic curves, there is a simpler question that we did not answer:

\begin{qu}
Let $\ell$ be a prime number.  If $E$ and $E'$ are $\ell$-isogenous elliptic curves, what is the proportion $P(E,E')$ of primes $p$ for which $E(\F_p) \simeq E'(\F_p)$?
\end{qu}

More precisely, we consider the ratio
\[
P(E,E')[X] = \frac{\lbrace p \leq X~:~ E(\F_p) \simeq E'(\F_p) \rbrace}{\pi(X)},
\]
where $\pi$ is the prime-counting function, and its limit $\lim_{X \to \infty} P(E,E')[X]$.  When this limit exists, we denote it by $P(E,E')$.   It is a consequence of our prior and current work that the limit always exists and can be computed using the Chebotarev Density Theorem. 

\subsection{Main Theorem} Let $E$ be an elliptic curve over $\Q$.  Fix an algebraic closure $\overline{\Q}$ of $\Q$ and write $\Gal_\Q$ for the absolute Galois group.  If $\ell$ is a prime number, then we write $T_\ell E$ for the $\ell$-adic Tate module of $E$ and 
\begin{align*}
&\rho_{E,\ell}: \Gal_\Q \to \Aut (T_\ell E), \text{ and} \\
&\overline{\rho}_{E,\ell^n}: \Gal_\Q \to \Aut (T_\ell E \otimes \Z/\ell^n\Z) 
\end{align*}
for the $\ell$-adic and mod $\ell^n$ representations of $E$, respectively.  If $G \subseteq \GL_2(\Zl)$ is the image of the $\ell$-adic representation, then we write $G(\ell^n) \subseteq \GL_2(\Z/\ell^n\Z)$ for its reduction modulo $\ell^n$. If $E'$ is $\ell$-isogenous to $E$ then we write  $G'$ and $G'(\ell^n)$ for the images of the $\ell$-adic and mod $\ell^n$ representations of $E'$, respectively.  Since $E$ and $E'$ are $\ell$-isogenous over $\Q$, it follows that $G(\ell)$ and $G'(\ell)$ are Borel subgroups of $\GL_2(\Z/\ell\Z)$.  The level  of the $\ell$-adic representation is $\ell^M$ where $M$ is the smallest positive integer for which $G$ is the full preimage of $G(\ell^M)$ in $\GL_2(\Z_\ell)$.

Let $L_{\ell^m} = \Q(E[\ell^m])$ and  $L_{\ell^m}' = \Q(E'[\ell^m])$ be the $\ell^m$-division fields of $E$ and $E'$, respectively.  Then $L_{\ell^m}$ and $L_{\ell^m}'$ are Galois over $\Q$ with Galois groups $G(\ell^m)$ and $G'(\ell^m)$, respectively.   If $p$ is a prime number, let $F$ and $F'$ denote matrix representatives of the Frobenius classes of $E$ and $E'$, respectively, as elements of $\GL_2(\Zl)$.  Then $F$ and $F'$ will define elements in $G(\ell^m)$ and $G'(\ell^m)$ by reduction modulo $\ell^m$. 

\begin{defn}
Let $m \geq 1$. We define $d_{\ell^m} \in [0, 1]$  to be the proportion of prime numbers such
that $F' \not \equiv I \pmod{\ell^m}$ given that $F \equiv I \pmod{\ell^m}$.  Similarly, we define $d'_{\ell^m}$ to be the proportion of prime numbers such
that $F \not \equiv I \pmod{\ell^m}$ given that $F' \equiv I \pmod{\ell^m}$. 
\end{defn}

In Section \ref{main_section} we determine the possible values of $d_{\ell^m}$ and $d'_{\ell^m}$ by interpreting conditional probabilities in terms of the degrees of the composite field extensions $[L_{\ell^m}L'_{\ell^m}:L_{\ell^m}]$ and $[L_{\ell^m}L'_{\ell^m}:L'_{\ell^m}]$.  Depending on the level of the $\ell$-adic images of $E$ and $E'$, once $m$ is sufficiently large it will follow from our work below that $d_{\ell^m} = d'_{\ell^m} = 1-1/\ell$.   In order to calculate $P(E,E')$, it is then a matter of determining $[L_{\ell^m}L'_{\ell^m}:L_{\ell^m}]$, $[L_{\ell^m}L'_{\ell^m}:L'_{\ell^m}]$, $|G(\ell^m)|$, and $|G'(\ell^m)|$ for small $m$ and then summing a geometric series.

If $E$ and $E'$ do not have CM, then $[\GL_2(\Zl):G] = [\GL_2(\Zl):G']$ by \cite[Prop.~2.1.1]{greenberg5}.  Therefore, if the maximum of the levels of their $\ell$-adic representations is $\ell^M$ then $|G(\ell^M)| = |G'(\ell^M)|$.  By the classification of non-CM $\ell$-adic images in \cite{rszb}, once $\ell >5$ the level of the $\ell$-adic representations of $E$ and $E'$ is $\ell$ (see also \cite{greenberg7}).  If the level is  $\ell$ then, as will follow from our work below, $P(E,E')$ is completely determined  by $|G(\ell)|$ and $|G(\ell')|$ (which must be equal).  If $\ell=5$ then there are a handful of special cases where the level is 25; here one would need to do separate calculations to determine $d_5$ and $d_5'$, but then the calculation of $P(E,E')$ would proceed similarly to the level $\ell$ case.  If $\ell=2$ or 3, then there are many possibilities for the 2- and 3-adic Galois images, as well as for the values of $d_{\ell^m}$ and $d_{\ell^m}'$, for small values of $m$ (we can always take $m \leq 5$).  However, these field extensions are small enough that one could, in principle, through routine computation, determine all possible $P(E,E')$ for all pairs of 2- and 3-isogenous elliptic curves over $\Q$ using the classification of 2-adic images in \cite{rzb} and 3-adic images in \cite{rszb}.   The general result on the values of the $d_{\ell^m}$ and $d_{\ell^m}'$ is as follows.

\begin{prop} \label{dmpropintro}
With notation as above, suppose the elliptic curves $E$ and $E'$ do not have CM and  that the maximum of the levels of their $\ell$-adic representations is $\ell^M$.  Then we have $d_{\ell^m}, d'_{\ell^m} \in \lbrace 0,1 - 1/\ell  \rbrace$ for all $m \geq 1$ and $d_{\ell^m} = d'_{\ell^m} = 1-1/\ell$ for all $m \geq M$.
\end{prop}

\begin{thm} \label{main_thm}
Let $E$ and $E'$ be $\ell$-isogenous elliptic curves over $\Q$.  Then the proportion $P(E,E')$ of primes such that $E(\F_p) \simeq E'(\F_p)$ is given by the formula
\begin{align} \label{main_form}
P(E,E') = 1-\sum_{m = 1}^\infty \left( \frac{d_{\ell^m}}{|G(\ell^m)|}  +  \frac{d'_{\ell^m}}{|G'(\ell^m)|}\right).
\end{align}
\end{thm}
 
\begin{rmk} \label{cmrmk}
Theorem \ref{main_thm} holds for all pairs of $\ell$-isogenous elliptic curves over $\Q$, including those with CM.  However, there are some significant differences in the CM case that make the evaluation of $P(E,E')$ much more straightforward.  We will show in Section \ref{CM} below that $P(E,E') =1$ once $\ell >3$.  If $\ell = 2$ or 3, then we can have $P(E,E') \ne 1$ but, as we will also show, there are only a handful of possibilities for $P(E,E')$.  This complements \cite[Thm.~1.11]{rnt} where we proved that if $E$ and $E'$ are 2-isogenous CM elliptic curves over $\Q$, then the proportion of anomalous primes is either $1/12$ or 0.
\end{rmk}

Because the techniques used in this paper are very similar to those of \cite{ck} and \cite{rnt}, we aim for brevity by referring the reader, whenever possible, to those background results rather than rederiving them here.  Throughout the paper we use the LMFDB \cite{lmfdb} for specific examples and Pari/GP \cite{PARI} for computation.

\section{Main Result} \label{main_section}

If $E$ and $E'$ are $\ell$-isogenous elliptic curves defined over a finite field $\F_p$, then the prime-to-$\ell$ parts of the groups $E(\F_p)$ and $E'(\F_p)$ are isomorphic \cite[Cor.~1]{cullinan1}.   Suppose the $\ell$-Sylow subgroup of $E(\F_p)$ has size $\ell^v$.  If $E(\F_p)$ and $E'(\F_p)$ are nonisomorphic, then by the theory of isogeny volcanoes (see \cite{sutherland} for background)  there exists a positive integer $m \leq v/2$ such that
\begin{align*}
E(\F_p)[\ell^\infty] &= \Z/\ell^{m}\Z \times \Z/\ell^{v-m} \Z, \text{ and} \\
E'(\F_p)[\ell^\infty] &= \Z/\ell^{m-1}\Z \times \Z/\ell^{v-m+1} \Z,
\end{align*}
or
\begin{align*}
E(\F_p)[\ell^\infty] &= \Z/\ell^{m}\Z \times \Z/\ell^{v-m} \Z, \text{ and} \\
E'(\F_p)[\ell^\infty] &= \Z/\ell^{m+1}\Z \times \Z/\ell^{v-m-1} \Z.
\end{align*}
In the first case, we have $F \equiv I \pmod{\ell^m}$ and $F' \not \equiv I \pmod{\ell^m}$, while in the second case we have $F'\equiv I \pmod{\ell^{m+1}}$ and $F' \not \equiv I \pmod{\ell^{m+1}}$.   
For the rest of the paper we fix a positive integer $m \geq 1$.   We now identify two proportions of prime numbers, stated in terms of the Frobenius classes to which they belong.

\begin{defn}
Let $X,X' \in [0,1]$ be defined as the following conditional probabilities:
\begin{align*}
X &= \Prop \left( F' \equiv I \pmod{\ell^m} ~|~ F \equiv I \pmod{\ell^m} \right),\text{ and} \\
X' &= \Prop \left( F \equiv I \pmod{\ell^m} ~|~ F' \equiv I \pmod{\ell^m} \right).
\end{align*}
\end{defn}

By the Chebotarev Density Theorem,
\begin{align*}
\Prop \left(F \equiv F' \equiv I \pmod{\ell^m} \right) &=  \frac{1}{[L_{\ell^m}L'_{\ell^m} :\Q]},  \\
\Prop \left(F \equiv I \pmod{\ell^m} \right) &= \frac{1}{[L_{\ell^m}:\Q]} = \frac{1}{|G(\ell^m)|}, \text{ and} \\
\Prop \left(F' \equiv I \pmod{\ell^m} \right) &= \frac{1}{[L'_{\ell^m}:\Q]} = \frac{1}{|G'(\ell^m)|}.
\end{align*}
By basic field theory 
\[
[L_{\ell^m}L'_{\ell^m} :\Q] = [L_{\ell^m}L'_{\ell^m}:L_{\ell^m}]\cdot |G(\ell^m)| = [L_{\ell^m}L'_{\ell^m}:L'_{\ell^m}]\cdot |G'(\ell^m)|.
\]

\begin{lem}\label{xfield}
With all notation as above, we have $X = [L_{\ell^m}L'_{\ell^m}:L_{\ell^m}]^{-1}$ and $X' = [L_{\ell^m}L'_{\ell^m}:L_{\ell^m}']^{-1}$.
\end{lem}

\begin{proof}
This is Bayes' Law applied to the Chebotarev Density Theorem and the field diagram
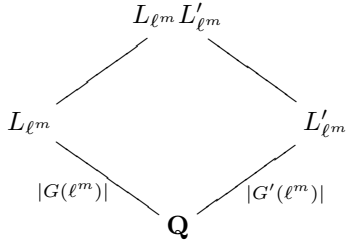
\[
\xymatrix{
&L_{\ell^m}L'_{\ell^m} \ar@{-}[dl] \ar@{-}[dr] \\
L_{\ell^m} \ar@{-}_{|G(\ell^m)|}[dr]& & L'_{\ell^m} \ar@{-}^{|G'(\ell^m)|}[dl] \\
&\Q
}
\]
\end{proof}

Since 
\begin{align*}
[L_{\ell^m}L'_{\ell^m} :\Q] &= [L_{\ell^m}L'_{\ell^m}:L_{\ell^m}]\cdot |G(\ell^m)| =  [L_{\ell^m}L'_{\ell^m}:L'_{\ell^m}]\cdot |G'(\ell^m)|,
\end{align*}
it follows that
\begin{align} \label{ratio}
\frac{[L_{\ell^m}L'_{\ell^m}:L'_{\ell^m}]}{[L_{\ell^m}L'_{\ell^m}:L_{\ell^m}]} = \frac{|G(\ell^m)|}{|G'(\ell^m)|} \in \lbrace 1,\ell,1/\ell \rbrace,
\end{align}
with $[L_{\ell^m}L'_{\ell^m}:L_{\ell^m}],  [L_{\ell^m}L'_{\ell^m}:L'_{\ell^m}] \in \lbrace 1,\ell \rbrace$.  For proofs of these statements, see \cite[\S 2]{rnt}.

\begin{lem} \label{dmx}
With all notation as above, we have $d_{\ell^m} = 1-X$ and $d_{\ell^m}' = 1-X'$.
\end{lem}

\begin{proof}
This follows directly from the definitions of $d_{\ell^m}, d'_{\ell^m}, X$, and $X'$. 
\end{proof}

\begin{prop} \label{geq5}
Suppose $E$ and $E'$ do not have CM and that the maximum of the levels of their $\ell$-adic representations is $\ell^M$.  Then if $m \geq M$, we have 
\[
[L_{\ell^m}L'_{\ell^m}:L_{\ell^m}] = [L_{\ell^m}L'_{\ell^m}:L'_{\ell^m}] = \ell.
\]
\end{prop}

\begin{proof}
Since the images of the $\ell$-adic representations of $E$ and $E'$ have level at most $\ell^M$, $G$ and $G'$ are the full preimages of $G(\ell^M)$ and $G'(\ell^M)$ in $\GL_2(\Z_\ell)$, respectively.  Fix $m \geq M$. We show that $[L_{\ell^m}L_{\ell^m}':L'_{\ell^m}] \ne 1$ by proving that there exists a prime $p$ for which $E(\F_p)[\ell^m] \cong \Z/\ell^m \Z \times \Z/\ell^m \Z$ but $E'(\F_p)[\ell^m] \not \cong \Z/\ell^m \Z \times \Z/\ell^m \Z$.  We follow a strategy similar to the proof of \cite[Theorem 5.2.4]{ck}.  We prove that there exists a prime $p$ for which the height of the $\ell$-isogeny volcano containing $E$ is exactly $m$, there is a unique vertex on the crater of the volcano, and this vertex corresponds to $E$.  In this case, the isogeny $E \to E'$ must be descending.  Switching the roles of $E$ and $E'$ then shows that $[L_{\ell^m}L_{\ell^m}':L_{\ell^m}] \ne 1$.

We consider odd $\ell$ and $\ell = 2$ separately.  Suppose $\ell$ is an odd prime and let $\delta$ be a quadratic nonresidue modulo $\ell$.  Since $\ell^m$ is at least the $\ell$-adic level of $E$, there exists a prime $p$ such that $F \equiv \left( \begin{smallmatrix} 1 & \ell^{m} \\  \delta \ell^m & 1 \end{smallmatrix} \right) \pmod{\ell^{2m+1}}$.  Since $t \equiv 2\pmod{\ell^{2m+1}}$ and $p \equiv \det(F) \equiv 1-\delta \ell^{2m} \pmod{\ell^{2m+1}}$, we see that 
\[
t^2 - 4p \equiv 4 - 4(1-\delta \ell^{2m}) \equiv 4 \delta \ell^{2m} \pmod{\ell^{2m+1}}.
\]
This implies $v_\ell(t^2-4p) = 2m$.  

Let $\mathcal{O}_0$ be the endomorphism ring of an elliptic curve lying on the crater of the $\ell$-isogeny volcano corresponding to $E,\ D_0 = \disc(\mathcal{O}_0)$, and $\mathcal{O}_K$ be the maximal order of $\Q(\sqrt{t^2-4p})$.  We now apply a result of Kohel to determine several important properties of this volcano \cite[Thm.~7]{sutherland}. Since $\ell \nmid [\mathcal{O}_K \colon \mathcal{O}_0]$, the height of this volcano is exactly $m$.  Since 
\[
\legen{D_0}{\ell} = \legen{4\delta}{\ell} = \legen{\delta}{\ell} = -1,
\] 
there is a unique vertex on the crater of the volcano.  Since $F \equiv I \pmod{\ell^{m}}$, we see that the vertex corresponding to $E$ must be this unique vertex on the crater of the volcano.  This completes the proof in the case where $\ell$ is an odd prime.

Now suppose $\ell = 2$.  Since $2^m$ is at least the $2$-adic level of $E$, there exists a prime $p$ such that $F \equiv \left( \begin{smallmatrix} 1 +2^m & 2^{m} \\   2^m & 1 \end{smallmatrix} \right) \pmod{2^{2m+3}}$.  Since $t \equiv 2 + 2^m \pmod{\ell^{2m+3}}$ and $p \equiv \det(F) \equiv 1 + 2^m - 2^{2m} \pmod{2^{2m+3}}$, we see that 
\[
t^2 - 4p \equiv 4 + 4\cdot 2^m + 2^{2m} - 4(1 + 2^m+ 2^{2m}) \equiv 5 \cdot 2^{2m} \pmod{2^{2m+3}}.
\]
This implies $v_2(t^2-4p) = 2m$. Applying \cite[Thm.~7]{sutherland} as we did for odd $\ell$, we see that the height of the $2$-isogeny volcano corresponding to $E$ is $m$, there is a unique vertex on the crater of the volcano, and this vertex corresponds to $E$.
\end{proof}

\begin{rmk}
In contrast to what was just proved, If $E$ and $E'$ have CM then by \cite[Prop.~6.1]{rnt}, the isogeny $E \to E'$ is either always ascending, always descending, or always horizontal, for each ordinary prime of good reduction.  In Section \ref{CM}  we will show that this implies that either $d_{\ell^m} = 0$ for all $m \geq 2$, or $d_{\ell^m}' = 0$ for all $m\geq 2$.  We also point the interested reader to \cite{adelman} for more background on the Galois theory of torsion point fields for non-CM curves
\end{rmk}

Proposition \ref{dmpropintro}, restated here as Proposition \ref{dmpropintro_restated} for convenience,  now follows as an immediate corollary of Proposition \ref{geq5} and Lemmas \ref{xfield} and \ref{dmx}. 

\begin{prop} \label{dmpropintro_restated}
With notation as above, suppose the elliptic curves $E$ and $E'$ do not have CM and  that the maximum of the levels of their $\ell$-adic representations is $\ell^M$.  Then we have $d_{\ell^m}, d'_{\ell^m} \in \lbrace 0,1 - 1/\ell  \rbrace$ for all $m \geq 1$ and $d_{\ell^m} = d'_{\ell^m} = 1-1/\ell$ for all $m \geq M$.
\end{prop}

We now prove the main result of the paper.

\begin{thm} \label{main_thm_restated}
Let $E$ and $E'$ be $\ell$-isogenous elliptic curves over $\Q$.  Then the proportion $P(E,E')$ of primes such that $E(\F_p) \simeq E'(\F_p)$ is given by the formula
\begin{align} \label{main_form2}
P(E,E') = 1 - \sum_{m = 1}^\infty \left( \frac{d_{\ell^m}}{|G(\ell^m)|}  +  \frac{d'_{\ell^m}}{|G'(\ell^m)|} \right).
\end{align}
\end{thm}

\begin{proof}
The summand $d_{\ell^m}/|G(\ell^m)| + d'_{\ell^m}/|G'(\ell^m)|$ represents the proportion of primes for which one of $E(\F_p)$ or $E'(\F_p)$ has full $\ell^m$-torsion, but the other does not.  Summing over all $m$ gives the proportion of primes for which $E(\F_p)$ and $E'(\F_p)$ are nonisomorphic, whence the formula (\ref{main_form}).
\end{proof}

Theorem \ref{main_thm_restated} holds regardless of whether $E$ or $E'$ has a rational point of order $\ell$.  If neither $E$ nor $E'$ has a rational point of order $\ell$, then at all primes for which neither $E(\F_p)$ nor $E'(\F_p)$ has a point of order $\ell$, we automatically have that $E(\F_p) \simeq E'(\F_p)$.  We therefore expect to see high values of $P(E,E')$ in such cases.  We conclude this section with an example illustrating this as well as a remark on quadratic twists.

\begin{exm}
Let $E$ be the elliptic curve \href{https://beta.lmfdb.org/EllipticCurve/Q/121/a/1}{121.a1} and $E'$ the elliptic curve \href{https://beta.lmfdb.org/EllipticCurve/Q/121/a/2}{121.a2}; they are 11-isogenous over $\Q$.  Each curve has level 11 and $d_{11^m} = d'_{11^m} = 10/11$ for all $m \geq 1$.  We also have $|G(11)| = |G'(11)| = 120$.  Applying Theorem \ref{main_thm_restated}, we have $P(E,E') = 86509/87840 \approx 0.9848$, while $P(E,E')[10^7] \approx 0.9839$.  

Checking the LMFDB, one sees that $E$ and $E'$ acquire an 11-torsion point over $\Q(\zeta_{11})^+$ and $\Q(\zeta_{11})$, respectively.  Note that $[\Q(\zeta_{11}):\Q] = 10$ and $[\Q(\zeta_{11})^+:\Q] = 5$.  We confirm in Pari/GP that $E(\F_p)$ and $E'(\F_p)$ have a point of order 11 for roughly 1/5 of the primes up to $10^7$ and thus  for the complementary 4/5 of the primes we automatically have $E(\F_p) \simeq E'(\F_p)$.
\end{exm}

\begin{rmk}
If $E$ admits a rational $\ell$-isogeny, then so does every quadratic twist $E_d$.  However, it is not the case that $P(E,E') = P(E_d,E_d')$, since twisting can change rational torsion.  For example, if $E$ is the elliptic curve \href{https://beta.lmfdb.org/EllipticCurve/Q/175/b/3}{175.b3} and $E'$ the curve \href{https://beta.lmfdb.org/EllipticCurve/Q/175/b/2}{175.b2}, then $E$ and $E'$ are 3-isogenous, yet neither has a rational 3-torsion point.  Using the information on the $3$-adic Galois images from their LMFDB pages, we compute $P(E,E') = 199/240$ via Theorem \ref{main_thm_restated}.

Twisting by $5$, we check that $E_5$ is the curve \href{https://beta.lmfdb.org/EllipticCurve/Q/35/a/3}{35.a3} and $E_5'$ is \href{https://beta.lmfdb.org/EllipticCurve/Q/35/a/2}{35.a2}, each of which has a rational 3-torsion point.  Using their LMFDB data we compute $P(E_5,E_5') = 79/120$.  We also note that $199/240 > 79/120$, reflecting the fact that neither $E$ nor $E'$ has a rational point of order 3 and so it is more likely that $E(\F_p) \simeq E'(\F_p)$ than $E_5(\F_p) \simeq E'_5(\F_p)$.
\end{rmk}

\section{Examples}

By our discussion following the proof of Theorem \ref{main_thm_restated}, we are interested in values of $P(E,E')$ that are not artificially ``inflated'' by neither $E$ nor $E'$ having a point of order $\ell$.  Therefore, suppose $E(\Q)_\tors \simeq \Z/\ell\Z$ and write $E' = E/\langle P \rangle$, where $P$ is a rational point of order $\ell$.  If $\ell =2$ then $E'$ always has a point of order $2$, while if $\ell  \in \lbrace 3,5 \rbrace$, then $E'$ has a rational point of order $\ell$ only when the image of the mod $\ell$ representation of $\ell$ lies in a split Cartan subgroup of $\GL_2(\Z/\ell\Z)$.  If $\ell=7$, then $E'$ will never have a point of order $\ell$ by the main result of \cite{greenberg7}.

Given this setup, the maximal order of $G(\ell)$ and $G'(\ell)$ is $\ell(\ell-1)$.   If $G$ and $G'$ have level $\ell$, then we have the following formula for $|G(\ell^m)|$ and $|G'(\ell^m)|$:
\[
|G(\ell^m)| = |G'(\ell^m)| = (\ell-1)\ell^{4m-3}.
\]
We also have $[L_{\ell^m}L'_{\ell^m}:L_{\ell^m}] = [L_{\ell^m}L'_{\ell^m}:L'_{\ell^m}] = \ell$ so that $d_{\ell^m} = d'_{\ell^m} = 1-1/\ell$.  Putting all of this together, when $G$ and $G'$ are as large as possible, we have
\begin{align*}
P(E,E') &=  1 - \sum_{m = 1}^\infty \left(\frac{d_{\ell^m}}{|G(\ell^m)|}  +  \frac{d'_{\ell^m}}{|G'(\ell^m)|}\right) \\
&= 1 - \frac{\ell-1}{\ell} \cdot \frac{1}{\ell(\ell-1)} \sum_{m = 0}^\infty \frac{1}{\ell^{4m}} \\
&= \frac{\ell^4 - 2\ell^2 -1}{\ell^4-1}.
\end{align*}
Define $f(\ell) \colonequals \frac{\ell^4 - 2\ell^2 -1}{\ell^4-1}$.  Then we compute
\begin{center}
\begin{tabular}{|ll|}
\hline
$\ell$ & $f(\ell)$ \\
\hline
2 & 7/15 = 0.4$\overline{6}$ \\
3 & 31/40 = 0.775\\
5 & 287/312 = 0.919$\overline{871794}$\\
7 & 1151/1200 = 0.9591$\overline{6}$\\
\hline
\end{tabular}
\end{center}

We now give some examples illustrating Equation (\ref{main_form}) when $E$ has a rational $\ell$-torsion point and $G$ is as large as possible, as well as a special case when $G$ is not.  In all of our examples we calculate out to $10^6$ and note that $\pi(10^6) = 78498$.

\begin{exm} For each of the following examples, $|G(\ell)| = |G'(\ell)| = \ell(\ell-1)$.  

\begin{enumerate}
\item Let $E$ and $E'$ be the 2-isogenous elliptic curves of the isogeny class with LMFDB label \href{https://beta.lmfdb.org/EllipticCurve/Q/69/a/}{69.a}.  We compute
\[
P(E,E')[10^6] = \frac{36631}{\pi(10^6)} \approx  0.46665 \approx 0.4\overline{6}  = f(2).
\]

\item Let $E$ and $E'$ be the 3-isogenous elliptic curves of the isogeny class with LMFDB label \href{https://beta.lmfdb.org/EllipticCurve/Q/44/a/}{44.a}.  
\[
P(E,E')[10^6] = \frac{72283}{\pi(10^6)} \approx  0.77548 \approx 0.775  = f(3).
\]

\item Let $E$ and $E'$ be the 5-isogenous elliptic curves of the isogeny class with LMFDB label \href{https://beta.lmfdb.org/EllipticCurve/Q/38/b/}{38.b}.  We compute
\[
P(E,E')[10^6] = \frac{60874}{\pi(10^6)} \approx  0.92082 \approx 0.919\overline{871794}  = f(5).
\]

\item Let $E$ and $E'$ be the 7-isogenous elliptic curves of the isogeny class with LMFDB label \href{https://beta.lmfdb.org/EllipticCurve/Q/26/b/}{26.b}.  We compute 
\[
P(E,E')[10^6] = \frac{75298}{\pi(10^6)}  \approx  0.95923 \approx 0.9591\overline{6}  = f(7).
\]
\end{enumerate}
\end{exm} 

\begin{exm}
Let $E$ be the elliptic curve with LMFDB label \href{https://beta.lmfdb.org/EllipticCurve/Q/144/b/4}{144.b4} and $E'$ be \href{https://beta.lmfdb.org/EllipticCurve/Q/144/b/3}{144.b3}.  Then $E(\Q)_{\tors} \simeq E'(\Q)_{\tors} \simeq \Z/2\Z \times \Z/2\Z$.  Therefore $d_{2} = d_{2}' = 0$.   According to their LMFDB pages, both elliptic curves have level 8, with $|G(4)| = 4$, $|G(8)| = 16$, $|G'(4)| = 8$, $|G'(8)| = 16$, and $L_4 = \Q(\zeta_{12})$.  Direct calculation with Pari/GP shows that $L'_4 = \Q(\zeta_{24})$ and thus
\begin{align*}
d_4 &=1 -  X = 1 - [L_4L_4':L_4]^{-1} = 1/2,\text{ and} \\
d_4' &=1 -  X' = 1 - [L_4L_4':L_4']^{-1} = 0.
\end{align*}
Further calculation reveals that $d_{2^m} = d_{2^m}' = 1/2$ for $m=3,4$.  By Theorem \ref{main_thm_restated} we have
\begin{align*}
P(E,E') &=1 -   \sum_{m = 1}^\infty \left(\frac{d_{2^m}}{|G(2^m)|}  +  \frac{d'_{2^m}}{|G'(2^m)|} \right) \\
&= 1 - \left(\frac{1/2}{|G(4)|} + \frac{1/2}{|G(8)|} \cdot \sum_{m=0}^\infty \frac{1}{16^m} + \frac{1/2}{|G'(8)|} \cdot \sum_{m=0}^\infty \frac{1}{16^m} \right)\\
&= 1-  1/8 - \frac{1}{32} \cdot \frac{16}{15} - \frac{1}{32} \cdot \frac{16}{15} = 97/120 = 0.808\overline{3}.
\end{align*} 
Finally, we compute
\[
P(E,E')[10^6] = \frac{63469}{\pi(10^6)}  \approx  0.8085.
\]
\end{exm}

\section{Remarks on CM Elliptic Curves} \label{CM}

By \cite[Thm.~1]{lenstra}, if $E$ is an ordinary elliptic curve over a finite field $\Fq$ with endomorphism ring $\mathcal{O}$, then $E(\Fq) \simeq \mathcal{O}/(\pi-1)$, where $\pi$ represents the Frobenius as an element of $\mathcal{O}$.    If $E$ and $E'$ are isogenous elliptic curves with endomorphism rings $\mathcal{O}$ and $\mathcal{O}'$ over $\Fq$, respectively, then the prime-to-$[\mathcal{O}:\mathcal{O'}]$ parts of the finite groups $E(\Fq)$ and $E'(\Fq)$ are isomorphic by \cite[Lem.~2]{cullinan1}.   We will study the effect of all of this in the case where $E$ and $E'$ define elliptic curves over $\Q$ with CM.  Before beginning, however, we recall that by \cite[Thm.~2.3.1]{ck} if $E$ and $E'$ have supersingular reduction then $E(\F_p) \simeq E'(\F_p)$.  Therefore, we can focus exclusively on primes of ordinary reduction.

 If $E$ and $E'$ have CM over $\Q$ then, by the Baker-Heegner-Stark Theorem, the CM orders correspond to one of the nine fundamental discriminants 
\[
-3,-4,-7,-8,-11,-19,-43,-67,-163,
\]
or one of the four non-fundamental discriminants, $-12,-16,-27,-28 $, of class number 1.  Suppose that $E$ and $E'$ have geometric endomorphism rings $\mathcal{O}$ and $\mathcal{O}'$ over $\Q$, respectively.  If $p > 3$ is a prime of ordinary reduction, then by \cite[Ch. 13, Thm. 12]{lang} $E \pmod{p}$ and $E' \pmod{p}$ have geometric endomorphism rings $\mathcal{O}$ and $\mathcal{O}'$ as well.  Because ordinary elliptic curves have all endomorphisms defined over the base field, we have that $\mathcal{O}$ and $\mathcal{O}'$ are the endomorphism rings over $\F_p$ of $E \pmod{p}$ and $E' \pmod{p}$, respectively. 

Since $E$ and $E'$ are isogenous, $\mathcal{O}$ and $\mathcal{O}'$ are homothetic and, as a result, $\disc \mathcal{O} / \disc \mathcal{O}'$ is a rational square.  Comparing against the possible discriminants above, those square values are either $1, 4, 9$, or their reciprocals.  If $\disc \mathcal{O} = \disc \mathcal{O}'$, then $\mathcal{O} \simeq \mathcal{O}'$.  This isomorphism holds modulo $p$ as well and thus $E(\F_p) \simeq E'(\F_p)$. 

\begin{lem}
Suppose $E$ and $E'$ are $\ell$-isogenous CM elliptic curves over $\Q$ with $\ell >3$.  Then $P(E,E') =1$.
\end{lem}

\begin{proof}
If $E$ and $E'$ are $\ell$-isogenous, then by \cite[Prop.~21]{kohel} the endomorphism rings, under reduction modulo ordinary $p$, satisfy $[\mathcal{O}:\mathcal{O}'] = 1,\ell$, or $\ell^{-1}$.  However, since $E$ and $E'$ have CM over $\Q$, by our work above it follows that $[\mathcal{O}: \mathcal{O}'] \in \lbrace 1,2^{\pm 1},3^{\pm 1} \rbrace$.  Therefore, we must have $\mathcal{O} \simeq \mathcal{O}'$ and hence $E(\F_p) \simeq E'(\F_p)$.  Since $E(\F_p) \simeq E'(\F_p)$ for all supersingular primes, it  follows that $P(E,E') =1$.
\end{proof}

We are now reduced to studying the cases $[\mathcal{O}:\mathcal{O}'] \in \lbrace 2^{\pm 1}, 3^{\pm 1} \rbrace$.  Suppose $\ell = 2$.  In \cite{rnt} we proved that in the case of CM elliptic curves, the proportion of anomalous primes of a pair of 2-isogenous curves is either 1/12 or 0.  This essentially came down to the fact (proved in  \cite[Prop.~6.1]{rnt}) that if $E \to E'$ is a 2-isogeny then either 
\begin{enumerate}
\item $L_{2^m} \subseteq L'_{2^m}$  for all $m \geq 2$, or 
\item $L_{2^m}' \subseteq L_{2^m}$  for all $m \geq 2$.
\end{enumerate}
Thus either $d_{2^m} = 0$ for all $m\geq 2$ or $d'_{2^m}=0$ for all $m \geq 2$.   It also follows from \cite[Prop.~6.1]{rnt} that if $d_{4} \ne 0$ (respectively $d_{4}' \ne 0)$, then $d_{2^m} = 1/2$ (resp.~$d_{2^m}' = 1/2$) for all $m\geq 2$.  The following example demonstrates how to use this property of CM elliptic curves to compute $P(E,E')$. 

\begin{exm}
Let $E$ and $E'$ be the 2-isogenous elliptic curves  \href{https://www.lmfdb.org/EllipticCurve/Q/49/a/1}{49.a1} and \href{https://www.lmfdb.org/EllipticCurve/Q/49/a/2}{49.a2}, respectively. These elliptic curves have CM, and one verifies that the geometric endomorphism ring of $E$ is $\Z[\sqrt{-7}]$, while that of $E'$ is $\Z[(-1 + \sqrt{-7})/2]$.  

According to their LMFDB pages, $L_2 = \Q(\sqrt{7})$ and $L_2' = \Q(\sqrt{-7})$.  Hence $d_2 = d_2' = 1/2$.  By \cite[Prop.~6.1]{rnt}, either $d_{2^m}=0$ for all $m\geq 2$ or $d'_{2^m} =0$ for all $m \geq 2$.  We check that $[L'_{4}:L_4] = 2$, hence $d'_{2^m} = 0$ for all $m\geq 2$ and $d_{2^m} = 1/2$ for all $m\geq 2$.   One checks that $|G(2^m)| = |G'(2^m)| = 2 \cdot 4^{m-1}$ for all $m\geq 2$.  Thus,
\begin{align*}
P(E,E') &= 1 - \sum_{m=1}^{\infty} \left( \frac{d_{2^m}}{|G(2^m)|} + \frac{d_{2^m}'}{|G'(2^m)|} \right) \\
&= 1 - \left( \frac{1}{2} \sum_{m=1}^\infty \left(\frac{1}{2\cdot 4^{m-1}}\right) + \frac{1}{2} \cdot \frac{1}{2} \right) \\
&= 1 - \left( \frac{1}{4} \cdot \frac{4}{3} + \frac{1}{4}\right) = \frac{5}{12} =0.41\overline{6}.
\end{align*}
We separately check that $P(E,E')[10^6] \approx 0.4165$.
\end{exm}

Similar to the non-CM case, one could work out all possible values of $P(E,E')$ for all  2-isogenous CM elliptic curves over $\Q$.  To finish the paper, we consider the case $\ell=3$.  By the classification of isogeny-torsion graphs over $\Q$ in \cite{chil} and the classification of $\ell$-adic Galois representations of CM elliptic curves over $\Q$, the isogeny-torsion graph 
\[
\xymatrix{
[3] \ar@{-}[r] & [3] \ar@{-}[r] & [3] \ar@{-}[r]  &[0]
}
\]
is the unique (up to isomorphism) isogeny-torsion graph over $\Q$ to which elliptic curves having CM discriminant $-27$ can belong.   All such isogeny-torsion graphs contain elliptic curves that are quadratic twists of those in the isogeny-torsion graph \href{https://beta.lmfdb.org/EllipticCurve/Q/27/a/}{27.a}.  Moreover, by \cite[Thm.~1.2, 1.4]{alvaro}, there are only a handful of cases where $E$ and $E'$ are 3-isogenous, have CM, and $\disc \mathcal{O}/\disc \mathcal{O}' \in \lbrace 3^{\pm 2} \rbrace$.  We list those cases in Table \ref{table1} according to their 3-adic images $G$ and $G'$, discriminants, $j$-invariants, and rational torsion subgroups.

\begin{center}
\begin{table}
\begin{tabular}{|l|l|c|}
\hline
$[G, \disc(\mathcal{O}), j,E(\Q)_\tors]$ & $[G', \disc(\mathcal{O}'), j',E'(\Q)_\tors]$ & Example\\
\hline
$[27.648.13.25,-27,-12288000, \Z/3\Z]$ & $[27.1944.55.31, -3, 0, \Z/3\Z]$ & \href{https://beta.lmfdb.org/EllipticCurve/Q/27/a/}{27.a} \\
\hline
$[27.648.13.34,-27,-12288000, 0]$ & $[27.1944.55.37, -3, 0, \Z/3\Z]$ & \href{https://beta.lmfdb.org/EllipticCurve/Q/27/a/}{27.a} \\
\hline
$[\text{maximal}, -27,$-12288000$,0]$ & $[27.972.55.16,-3,0,0]$&\href{https://beta.lmfdb.org/EllipticCurve/Q/432/e/}{432.e} \\
\hline
\end{tabular} \caption{Elliptic Curves with CM discriminant $-27$} \label{table1}
\end{table}
\end{center}

We conclude by working out $P(E,E')$ for an example of the third case of Table~\ref{table1}.  

\begin{exm}
Suppose $E$ is \href{https://beta.lmfdb.org/EllipticCurve/Q/432/e/1}{432.e1} and $E'$ is  \href{https://beta.lmfdb.org/EllipticCurve/Q/432/e/4}{432.e4}.  From their LMFDB pages, one checks that $L_3 = \Q(\zeta_{12})$, which is an index-3 subfield of $L_3'$.  Therefore $d_3 = 2/3$ and $d_3'=0$.  By direct calculation, $d_9 = 2/9$ and $d_9'=0$.  The proof of \cite[Prop. 6.1]{rnt} carries over immediately to the case $\ell=3$ as well (since the endomorphism rings of $E \pmod{p}$ and $E' \pmod{p}$ are fixed for all ordinary $p$, the rational 3-isogeny $E \to E'$ always reduces to an ascending 3-isogeny), hence $d_{3^m}' = 0$ for all $m$ and $d_{3^m} = 2/3$ for all $m$.  We have $|G(3^m)| = 4\cdot 9^{m-1}$ for all $m \geq 1$.  Now we compute:
\[
P(E,E') = 1 - \sum_{m = 1}^\infty \left( \frac{d_{3^m}}{|G(3^m)|} +  \frac{d'_{3^m}}{|G'(3^m)|} \right) = 1-\frac{2}{3} \cdot \frac{1}{4} \cdot \frac{10}{9}  = \frac{22}{27} = 0.\overline{814}.
\]
For comparison, a routine calculation shows that $P(E,E')[10^7] \approx 0.8126$.
\end{exm}

\section*{Acknowledgments}

The second author was supported by DMS 2154223.

\end{document}